%% file: CRW.tex
\documentclass{amsart}
\usepackage{bm}
\usepackage{amssymb}
\usepackage{amsthm}
\usepackage{microtype}
\usepackage{tikz}
\usepackage{framed}
\usepackage{theoremref}
\usepackage{comment}
\usepackage{subfigure}

\usepackage{rotating}
\usepackage{changepage}
\usepackage[showframe=false, margin = 1.5 in]{geometry}
\usepackage{paralist}
\usepackage{enumitem}

\renewcommand{\emptyset}{\varnothing}
\newcommand{\E}{\mathbf{E}}
\renewcommand{\P}{\mathbf{P}}

\newcommand{\se}{\subseteq}

\renewcommand{\emptyset}{\varnothing}

\newcommand{\Z}{\mathbb Z}

\newcommand{\inv}{^{-1}}
\newcommand{\f}{\frac}
\newcommand{\ind}[1]{\mathbf{1}{\{#1\}} }

\usepackage{color}
\definecolor{dark_green}{RGB}{1, 180, 1}

\usepackage{mathrsfs}

\theoremstyle{plain}
\newtheorem{thm}{Theorem}
\newtheorem{lemma}[thm]{Lemma}
\newtheorem{prop}[thm]{Proposition}

\newtheorem{cor}[thm]{Corollary}

\newtheorem{conjecture}[thm]{Conjecture}

\theoremstyle{remark}
\newtheorem{remark}[thm]{Remark}

\begin{document}

\input{intro3}

\input{fq2}
\input{main2}

\input{variants2}
\bibliographystyle{amsalpha}
\bibliography{CRW}

\end{document}

%% file: intro3.tex
\title{Site recurrence for coalescing random walk}

\author[I. Benjamini]{Itai Benjamini}
\address{Department of Mathematics, University of Washington}
\email{itai.benjamini@weizmann.ac.il}

\author[E. Foxall]{Eric Foxall}
\address{Department of Mathematical and Statistical Sciences, Arizona State University}
\email{Eric.Foxall@asu.edu}

\author[O. Gurel-Gurevich]{Ori Gurel-Gurevich}
\address{Einstein Institute of Mathematics, Hebrew University of Isreal}
\email{origurel@math.huji.ac.il}

\author[M. Junge]{Matthew Junge}
\address{Department of Mathematics, University of Washington}
\email{jungem@math.washington.edu}

\author[H. Kesten]{Harry Kesten}
\address{Department of Mathematics, Cornell University}
\email{kesten@math.cornell.edu}

\maketitle

\begin{abstract}
Begin continuous time random walks from every vertex of a graph and have particles coalesce when they collide. We use a duality relation with the voter model to prove the process is site recurrent on bounded degree graphs, and for Galton-Watson trees whose offspring distribution has exponential tail. We prove bounds on the occupation probability of a site, as well as a general 0-1 law. Similar conclusions hold for a coalescing process on trees where particles do not backtrack.

\end{abstract}

\section{Introduction}

\emph{Coalescing random walk} (CRW) 
starts with one particle at each vertex of an undirected graph. Each then performs a continuous time nearest neighbor random walk, jumping according to a mean 1 exponential clock. 
When particles collide they bind together and proceed as one. 
Say that CRW is \emph{site recurrent} if every site is almost surely visited infinitely often. If instead this occurs with probability 0, call the process \emph{transient}. Our main tool for proving site recurrence is the following necessary and sufficient condition in terms of the expected occupation time of a vertex.

\begin{prop}\thlabel{thm:main}
Site recurrence 
 is equivalent to infinite expected occupation time at any vertex.
Moreover, CRW is either site recurrent or transient (i.e.\ it satisfies a 0-1 law). 
\end{prop}

Let $p_t(v)$ be the probability a particle is at the vertex $v$ at time $t$, so that site recurrence is equivalent to divergence of $\int p_t(v)$. 
We use duality with the voter model to obtain non-integrable lower bounds on the following graphs:
%
\begin{thm}\thlabel{thm:bgw} CRW is site recurrent on:
\begin{enumerate}[label = (\roman*)]
    \item \label{bounded} Bounded degree graphs. If the maximum degree is $D$, then for all vertices $v$  $$p_t(v) \geq (1+ Dt)^{-1}$$
     for all $t \geq 0.$
    \item \label{gw} Galton-Watson trees with offspring distribution on $\mathbb Z^+$ and the probability of $k$ offspring bounded by $e^{-ck}$ for some $c>0$ and large enough $k$.
	 Here $$p_t(v) \geq C (t \log t)^{-1}$$ for all vertices $v$, large enough $t$, and some $C>0$ that depends on $c$.
\end{enumerate}
\end{thm}
Note that there are unbounded degree graphs for which CRW is not site recurrent; even the non-coalescing system of independent random walks is transient on trees with rapidly increasing degree. We are not sure how much the exponential tail hypothesis in (ii) can be weakened. See Further Questions (a) for more discussion.
A corollary to \thref{thm:main} is a general upper bound on the probability that a vertex $v$ is unoccupied on the interval $(t,u)$.
 
 \begin{cor} \thlabel{cor:ub}
Let $\sigma_t$ be the first time after $t$ that $v$ is occupied by a particle. It holds that
$$\P(\sigma_t > u) \leq \f{ t }{ t + \int_t^u p_s(v) ds}.$$
And, for a graph with maximum degree $D$ 
$$\P(\sigma_t >u ) \leq \f{ t }{ t + \f 1 D (\log(1 + Du) - \log( 1 + Dt)) } = O(1/\log(u)).$$
 \end{cor}

We also give a universal upper bound for $p_t(v)$ on general graphs. It follows that the occupation probability decays to 0 for any graph. The upper bound is a small modification of an argument in \cite{griffeath}, so we also credit David Griffeath.

\begin{prop}[Griffeath]
\thlabel{prop:pt}
Let $G$ be a connected, infinite graph. For all vertices $v$, any $\epsilon>0$, and large enough $t$ (depending on $\epsilon$) it holds that $p_t(v) \leq (1+\epsilon) / (2\sqrt{ \pi t})$. 
\end{prop}


%
\subsection*{History}
The study of coalescing systems began in the 1970's with the paper of Erd\H{o}s and Ney \cite{erdos1974}. The duality relationship to the voter model, which we rely heavily upon, was first observed in \cite{holley1}. Variations of coalescing random walk 
continue to find new applications. For example, random coalescence involving multiple types of particles, and particle interaction rules, is used to model certain chemical reactions (see \cite{holley}, \cite{kesten_types} and \cite{kesten_number}). Non-spatial models such as Kingman coalescence (\cite{kingman}) find applications in modeling ancestry in biology. 
A survey of coalescence models can be found in \cite{overview}.
Arratia \cite{arratia} looks at site recurrence for discrete time walks, and annihilating systems, both with possibly vacant sites in the starting configuration.
CRW is applied to the voter model in \cite{balaz_voter}. Also, it is studied in more generality in \cite{double_annihilating} and \cite{Z_coalescing}.
Other recent articles have focussed on different settings.
Its behavior on finite graphs is of interest to computer scientists. The model on the $d$-dimensional torus is introduced in \cite{cox}. There, they study the expected time for the process to coalesce into a single particle. In \cite{cs} the coalescence time is studied on a variety of finite graphs. Elsewhere, in a continuous spatial setting, recurrence is studied with coalescing diffusions by Cabezas, Rolla and Sidoravicious in \cite{2013arXiv1309.4387C}.

Early results for coalescing random walk focused on the lattices $\Z^d$. In \cite{griffeath} Griffeath shows that both coalescing and annihilating random walk on $\Z^d$ is a.s.\ \emph{weakly recurrent}, under certain restrictions on the vacant sites in the starting configuration. Weakly recurrent means that each site is occupied infinitely often, but for a decreasing fraction of time. An important ingredient in the proof of recurrence is an estimate for the function $p_t$, the probability a particle occupies the origin at time $t$. In \cite{first}, Bramson and Griffeath study $p_t$ in the coalescing case and compute its asymptotics for every $d\geq 1$. Rather nicely, for $d \geq 3$ it holds that $p_t \sim (\gamma_dt)\inv$ with $\gamma_d$ the probability a random walk on $\mathbb Z^d$ never returns to its starting point. The proof of this is computational; later, \cite{kesten} Kesten gives a probabilistic argument that revolves around the heuristic $p_t' \approx -\gamma_d p_t^2$.

\subsection*{Overview}
The main idea 
 is to obtain information about $p_t(v)$ from a dual voter model. This dual was first applied to CRW in \cite{holley1}, and subsequently utilized in \cite{hs, griffeath, arratia2, arratia}. 
 In \thref{cor:pt} we deduce that $p_t(v)$ is equivalent to the probability a time-changed nearest neighbor simple random walk avoids 0 up to time $t$. All of our estimates come from studying this random walk. 

\subsection*{Further Questions} \label{questions} 
We record several questions regarding coalescing and annihilating random walk here:
\begin{enumerate}[label = (\alph*)]

\item \label{q:GW} Can the assumptions on the degree 
 in \thref{thm:bgw} (ii) be weakened? We expect that our approach extends (at least) to stationary graphs with finite expected degree.



	\item {Suppose $G$ is an  infinite  unimodular random graph in which  each vertex  is assigned an  infinite
trajectory in an ergodic invariant way (see \cite{ergodic}).  Particles, one from each vertex,
start moving along their trajectory in continuous time and annihilate when meeting.}
Is  the resulting process recurrent? \emph{Start by showing it on Euclidean lattices.}

	\item {Place $\epsilon$-balls (meteors) in Euclidean space with centers according to a unit intensity Poisson process. At time 0 each chooses a direction uniformly randomly and proceeds along this direction at unit speed (non-random). When two meteors collide, they annihilate.} Is the origin a.s.\ occupied by infinitely many meteors for all $d \geq 1$? {This is discussed in more detail in Section \ref{sec:trees}.}
\item {Have particles perform annihilating random walk on a graph where the particle started at $x$ steps according to an exponential clock with mean an independent uniform$(0,1)$ random variable.}
	Is this model on $\Z^d$ still recurrent? If so, what can be said of the limiting speed of the particles visiting the origin? {Possibly slower moving particles survive longer, and the average speed of particles visiting the origin decays with time.}

\end{enumerate}

\subsection*{Outline}

Section \ref{sec:main} starts with the proof of \thref{thm:main}. We also establish, in \thref{lem:rw}, that infinite expected occupation time is equivalent to survival of a nearest neighbor random walk in the voter model. \thref{cor:pt} relates this back to $p_t$. Sections \ref{sec:bounded} and \ref{sec:gw} contain the proofs of \thref{thm:bgw} (i) and (ii), respectively. Section \ref{sec:trees} discusses some non-backtracking variants, and contains the proof of site recurrence for a non-backtracking model on bounded degree trees and Galton-Watson trees with exponential tail.

%% file: main2.tex
\section{Site recurrence for coalescing random walk} \label{sec:main}

Coalescing random walk on a graph, $G= (V,E)$, has a graphical representation as follows: each edge is replaced with two directed edges and an independent Poisson process with unit intensity is placed on each directed edge, indexed by time. When the bell of a Poisson process for the edge $(u,w)$ rings we check if there's a particle at $u$ and if so, we move it to $w$. If there's already a particle at $w$, they merge. We denote the process $(\xi_t)_{t\geq 0}$ with $\xi_t\in \{0,1\}^V$ equal to the set of occupied vertices at time $t$, and occasionally $\xi_t^v$ for the location at time $t$ of the particle that began at $v$. 
In this notation we have $p_t(v) = \P(\xi_t(v)=1)$ is the probability that $v$ is occupied by a particle at time $t$. Thus, $\int_0^\infty p_t(v) dt$ is the expected occupation time of $v$. 


\begin{proof}[Proof of \thref{thm:main}]
If there is positive probability of infinite occupation time at $v$, then the expected occupation time is infinite. For the other direction we generalize 
\cite[Lemma 2]{arratia}.

 Suppose that $\int_0^{\infty}p_t(v) = \infty.$ For any $t \in [0, \infty)$ let $\sigma_t = \inf \{ s \geq t \colon  \xi_t(v) =1\} \in [0,\infty]$, the first time after $t$ that $v$ is occupied by a particle. We wish to establish that
$$\P(\sigma_t < \infty) =1, \qquad \forall t\geq 0.$$
The basic coupling $\xi_t^A = \{ \xi_t^x \colon x \in A\}$ for $A \subseteq V$ has the property that $A \subseteq B \subseteq V$ implies $\xi_t^A \subseteq \xi_t^B \subseteq \xi_t$, so the Markov process $(\xi_t^A \colon A \subseteq V)$, with state space $\{ 0,1\}^V = \{ A \colon A \subseteq V\}$ ordered by set inclusion, is attractive. Define $p_t^A = \P(\xi_t^A(v) =1),$ and also $I(t,u)= \int_{t}^u p_s(v) ds.$
  By assumption,
\begin{align}
\lim_{u \to \infty} I(t,u) = \infty, \qquad \forall t \geq 0. \label{eq:limit}
\end{align}

Let $f_{\sigma_t}$ be the density function of $\sigma_t$ and $\E_{\sigma_t}$ denote the expectation taken over all possible realizations of $\xi_{\sigma_t}$ given that $\sigma_t = r$.
 Using the strong Markov property we have for $t< u$,
\begin{align*}
I(t,u)  := \E \int_t^u \ind{ \xi_s(v) =1 } ds &= \int_{r=t}^u f_{\sigma_t}(r) \E_{\sigma_t}\left( \E \int_0^{u-r} \ind{\xi_s^A (v) =1} ds \right) dr \\
&\leq  \int_{r=t}^u f_{\sigma_t}(r) \left(\E \int_0^{u} \ind{ \xi_s(v) =1 } ds\right) dr\\
&\leq \P(\sigma_t \leq u) (t+I(t,u)).
\end{align*}
Dividing both sides by $t + I(t,u)$ we arrive at the inequality
\begin{align}\P ( \sigma_t \leq u ) \geq I(t,u) / ( t+I(0,u)) \label{eq:lb}.
\end{align}
 For fixed $t$, taking $u \to \infty$ yields $\P (\sigma_t < \infty) \geq 1$ by \eqref{eq:limit}. 

We conclude by describing the $0-1$ law. The above argument establishes that if CRW occupies a site for infinite time with positive probability, then it does so with probability 1. As $G$ is assumed to be connected, it follows that all sites are occupied infinitely often with probability 1. Therefore, the process is either site recurrent (recall this is defined as an almost sure event) or transient.  
\end{proof}
\begin{proof}[Proof of \thref{cor:ub}]
The lower bound on $\P(\sigma_t \leq u)$ at \eqref{eq:lb} yields an upper bound on the probability $v$ is unoccupied from time $t$ to $u$:
\begin{align}\P( \sigma_t > u )  \leq 1 - \f{ I(t,u) }{ t+ I(t,u) } =
\f{ t }{ t+ \int_t^u p_s(v) ds}. \label{eq:ub}
\end{align}
Which is the first part of the corollary. The second part follows by applying the bound on $p_t(v)$ in \thref{thm:bgw} (i) and integrating.
\end{proof}



\thref{thm:bgw} allows for site recurrence to be deduced by proving $p_t(v)$ is non-integrable. 
Our approach is to express $p_t(v)$ in another way. Consider the dual process to this model, which is called the \emph{voter model}. In the dual model we start with a partition of the space into clusters, where initially each vertex corresponds to a different cluster. When the bell at $(u,w)$ rings the vertex $w$ is added to the cluster containing $u$. We denote the process $(\zeta_t^v)_{t\geq 0}$ where $\zeta_t^v$ is the set of vertices belonging to the cluster that initially consists of the vertex $v$. If we run this model in reverse time, from time $t$ to $0$, we see the cluster that began at $v$ at time $t$ at time $0$ consists of exactly the particles that in the coalescing model are at $v$ at time $t$. In particular,
\begin{align}
p_t(v) = \P(\xi_t(v)=1) = \P(\zeta_t^v \neq \emptyset).
\label{eq:dualrel}
\end{align}

The advantage of working with the voter model is that the size of $\zeta_t^v$ is a nearest-neighbor symmetric random walk with transition rate depending on the number of boundary edges. Indeed, at any moment there are some directed edges going out of the cluster, and the same number of edges coming in. More precisely, the cluster size is a skip-free process on the integers 
that moves with rate equal to the size of the current boundary of the cluster. We record this fact in the following lemma. Let $|\cdot|$ denote either the counting measure of a finite set.

\begin{lemma} \thlabel{lem:rw}
Define $\zeta_t^v \se V$ to be the set of vertices in the cluster of $v$ at time $t$. 
 For each $x \in \zeta_t^v$ let $\partial_t(x) = \{ (x,y) \in E\colon y \notin \zeta_t^v\}$. Let $ \sum_{x \in \zeta_t^v} |\partial_t(x)|$ be the number of edges leading out of $\zeta_t^v$. The process has the following properties:
\begin{enumerate}[label = \emph{(\roman*)}]
    \item $|\zeta_{0}^v| = 1$.
    \item Let $ \tau = \inf\{t \colon |\zeta_t| = 0\}$. For all $t \geq \tau$ it holds that $|\zeta_t^v | = 0$.
    \item 
    The process is a martingale that transitions to $|\zeta_t^v| \pm 1$ at rate $\sum_{x \in \zeta_t^v} |\partial_t(x)|$.
\end{enumerate}

\end{lemma}

\begin{proof}
Properties (i) and (ii) follow from the construction of $\zeta_t^v$. We turn our attention to property (iii). For each $x \in \zeta_t^v$, $x$ is removed from $\zeta_t^v$ at rate $|\partial_t(x)|$ and each of the $|\partial_t(x)|$ sites in $\partial_t(x)$ is added to $\zeta_t^v$ at rate $1$. Since the rates balance, $$\E [|\zeta_t^v| \mid |\zeta_{t_-}^v|] = |\zeta_{t_-}^v|,$$ which establishes $|\zeta_t^v|$ is a martingale.  Summing the rates over $x \in \zeta_t^v$ shows that $|\zeta_t^v|$ transitions to $|\zeta_t| + 1$, and to $|\zeta_t^v|-1$, each at rate $\sum_{x \in \zeta_t^v} |\partial_t(x)|$. 
\end{proof}

This lets us describe $p_t$ in terms of the voter model.

\begin{cor} \thlabel{cor:pt}
    $p_t(v) = \P( |\zeta_t(v)| >0 ).$ This is the probability a nearest neighbor random walk with transition rate $\sum_{x \in \zeta_t} | \partial_t(x) |$ and absorbing state at 0 is yet to reach zero at time $t$. 
\end{cor}

 It follows that $p_t \to 0$ on any infinite, connected graph. 


\begin{proof}[Proof of \thref{prop:pt}] 
The transition rate in $\zeta_t^v$ is always at least two.
 By \thref{cor:pt} and a straightforward coupling we have $p_t$ is at most $\tilde{p}_t = \P(X_s>0,\,\,\forall s\leq t)$, with $X_s$ a rate-2 continuous time simple symmetric random walk started at $1$. Using the reflection principle together with the local central limit theorem, $\tilde{p}_t \sim 1/(2\sqrt{ \pi t})$ as $t\rightarrow\infty$, and the result follows. 
\end{proof}


\begin{remark} For coalescing walk on $\Z$ with nearest neighbour connections, since $\zeta_t^v$ is always of the form $\{x,x+1,...,x+k\}$ for some $x \in \Z,k \in \Z^+$, its transition rate is exactly $2\cdot 2 = 4$, so the above inequality is an equality, and gives the exact asymptotics $p_t \sim 1/(2\sqrt{\pi t})$, as observed in \cite{first}. Compared to \cite{first} there is an extra factor of $1/2$; our convention differs from theirs in that the transition of a particle at $v$ is equal to $\deg v$ and not $1$, since in our case $\deg v$ is allowed to vary.
\end{remark}

\subsection{Site recurrence for bounded degree graphs} \label{sec:bounded}
Now we turn our attention to proving site recurrence on general graphs. Define $\tau_v = \inf\{t\colon \zeta_t^v = \emptyset\}$. Integrating over $t$ in the duality relation \eqref{eq:dualrel} we find
$$\E \tau_v = \int_0^{\infty}\P(|\zeta_t^v|>0)dt = \int_0^{\infty}\P(\xi_t(v)=1)dt = \int_0^{\infty}p_t(v)dt.$$
So, proving site recurrence is equivalent to showing that the first hitting time of 0 for the simple random walk $|\zeta_t^v|$ (with random and time-varying transition rate) has infinite expectation. We start with the case when $G$ has bounded degree.
 
\begin{proof}[Proof of \thref{thm:bgw} (i)]
Let $v \in V$ with the maximum degree of vertices in $G$ bounded by $D$. \thref{lem:rw} establishes that the transition rate of $\zeta_t^v$ is less than or equal to $D|\zeta_t^v|$. Let $W_t$ be a continuous time nearest-neighbour random walk on $\Z^+ \cup \{0\}$ with $W_0=1$. The walk transitions from $k\in \mathbb Z^+ \cup \{0\}$ to $k \pm 1$ each at rate $Dk$, and is absorbed at $0$. Letting $\theta(t) = \P(W_t>0)$, if follows from \thref{cor:pt} and a straightforward coupling of $|\zeta_t^v|$ with $W_t$ that $p_t(v) \geq \theta(t)$, so it suffices to control $\theta(t)$.

The process $W_t$ can be interpreted as the number of particles in a branching process in which each particle independently dies, or gives birth to a single offspring, each at rate $D$. Let $\rho(t) = \P(W_t =0)$. By independence, for each $k\geq 0$ we have 
\begin{align}
\P(W_{t+h}=0 \mid W_h = k) = \rho(t)^k. \label{eq:rhok}
\end{align}
Recalling that $W_0=1$ then conditioning on $W_h$ for small $h>0$,
\begin{align*}
\rho(t+h) &= \P(W_{t+h}=0) = \sum_{k\geq 0}\P(W_{t+h}=0 \mid W_h=k)\P(W_h=k).
\end{align*}

The event that two or more transitions happens on the interval $[0,h]$ is contained in the event that a rate $2D$ exponential clock rings, then a rate $4D$ exponential clock rings (i.e.\ we go from 1 to 2 particles then another transition happens). The probability of this is bounded by the probability that $X+Y < h$ for $X$ and $Y$ rate $4D$ exponential random variables, and has density $f_{X+Y}(t) = \lambda^2 t e^{ - \lambda t }$ with $\lambda = 1/4D$. Integrating on $[0,h]$, then taking the Taylor expansion we have $\P(X+Y < h ) = O(h^2)$.
Since we will be dividing by $h$ and letting it tend to 0, we can combine all of the events that occur with two or more transitions into an $O(h^2)$ term. Using the expression \eqref{eq:rhok} this lets us write $\rho(t + h)$ as
\begin{align*}
\rho(t+h)&= \underbrace{1 \cdot Dh}_{\text{dies out}} + \underbrace{\rho(t)(1-2Dh)}_{\text{no change}} + \underbrace{\rho(t)^2Dh}_{\text{increases by 1}}  + \;\;O(h^2).
\end{align*}

Subtracting $\rho(t)$, dividing by $h$ and taking $h\downarrow 0$ this converges to the equation $\rho'=D(1-\rho)^2$. So, for the survival probability $\theta(t) =1-\rho(t)$ we find $\theta' = -\theta^2$, with $\theta(0)=1$, whose unique solution is $\theta(t) = 1/(1+Dt)$.
\end{proof}

\subsection{Site recurrence for Galton-Watson trees} \label{sec:gw}

A more general upper bound on the transition rate is $$|\text{maximum exposed degree}|\cdot|\zeta_t|.$$ Our hypothesis that the offspring distribution of our Galton-Watson tree has exponential tail guarantees that the maximum exposed degree is asymptotically bounded by $\log(\text{number of steps}).$ Ultimately this lets us compare with the divergent integral $\int_{t_0}^\infty  (t \log t)^{-1} dt.$ This is made rigorous below.

\begin{proof}[Proof of \thref{thm:bgw} (ii)]

Again, by \thref{cor:pt} it suffices to prove that $\zeta_t^v$ has infinite expected survival time. For convenience we denote $\zeta_t^v$ by $\zeta_t$.
Let $H_t = \cup_{s \leq t}\zeta_t$ be the vertices visited up to time $t$. Define the random times $0=t_1<t_2<...$ as when a vertex is added to $H_t$, and list them as $v_1,v_2,...$ in the order they are discovered, with $v_1$ being the root ($\rho$). The transition rate of $|\zeta_t|\rightarrow |\zeta_t| \pm 1$ is at most $M_t|\zeta_t|$ where
$$M_t = \sup\{\deg v\colon v \in H_t\}.$$
So, our first goal is to construct the voter model in such a way that $M_t$ can be easily controlled. A simple way to do this is to construct $G$ ``on the fly.'' That is, let $(X_i)_{i \geq 1}$ be an i.i.d.\ sequence of copies of the offspring distribution, and at time $t_i$, sample $X_i$ to determine the offspring distribution of $v_i$, which is then fixed for all $t>t_i$. This does not disturb the sample path distribution of $\zeta_t$, and has the advantage that the quantity
$$D_k = \max_{i \leq k} \deg v_i$$
is equal to $\max_{i \leq k}X_i$ where $X_i$ is a \emph{fixed} (as opposed to being a randomly indexed) i.i.d.\ sequence. Since, by assumption, $\P(X_i > x) \leq e^{-cx}$ for some $c>0$ and large enough $x$, a union bound gives $\P(D_k > x) \leq ke^{-cx}$, and setting $x = (3/c)\log k$,
\begin{align} \P(D_k > (3/c)\log k) \leq k^{-2} \label{eq:logmax} \end{align}
for large enough $k$.

Now, let $0=t_0<t_1<t_2<...$ denote the jump times of $\zeta_t$. In what follows we will want the set of jump times to be infinite, so if $\zeta_{t_i}=0$ (i.e.\ the cluster dies out), just include jumps back to 0 at rate 1. Since $|\xi_i|$ is a martingale with $\E|\zeta_{t_i}|=|\zeta_{0}|=1$ for all $i$, Doob's martingale inequality implies that for all $n > 0$
\begin{align}\P(\sup_{i}|\zeta_{t_i}|>bn) \leq (bn)^{-1}\label{eq:doob} \end{align}
for any $b>0$. Clearly $M_{t_i}$ is nondecreasing in $i$ and $M_{t_i} \leq D_{1+i}$, since vertices are exposed one at a time. Thus, the transition rate of $|\zeta_{t_i}|$ is at most $D_{i+1} |\zeta_{t_i}|$. Combining these observations with \eqref{eq:logmax} and \eqref{eq:doob}, we find that with probability at least $1 - (bn)^{-1} - n^{-2}$, for $t \leq t_{n^2}$ the transition rate in $\zeta_t$ is at most 
\begin{align}(3/c)\log(n^2+1) b n \leq (6b/c)n\log(n+1).\label{eq:ratebound} \end{align}
Let $m_n = m_n(b,c) = (6b/c)n \log(n+1)$. A quick summary: with high probability the first $n^2$ transitions happen at rate no more than $m_n$. Equivalently, the time $t_{n^2}$ is bounded below by the sum of $n^2$ independent exponentials with rate $m_n$. This is an Erlang distribution, $X_{n^2}$, with shape parameter $n^2$ and rate $m_n$. Thus, we have
$$\text{mean: } \mu_n = \f{n^2}{ m_n}, \text{ and variance: } \sigma_n = \f{\mu_n}{m_n}.$$
Chebyshev's inequality guarantees that 
\begin{align*}
\P(|X_n - \mu_n|\geq \mu_n/2) \leq \f{ \mu_n/m_n} { (\mu_n/2)^2 } = \f { 4 }{\mu_n m_n} = \f 4 {n^2}.
\end{align*}
One side of the above estimate is  
\begin{align}\P( X_{n^2} \leq \mu_n/2 ) \leq 4 n^{-2}\label{eq:erlang}.
\end{align}
By comparison, and using \eqref{eq:ratebound} and \eqref{eq:erlang}, we have
$$\P(t_{n^2} \geq n^2/(2m_n) ) \geq 1 - (bn)^{-1} - n^{-2} - 4 n^{-2}. $$
From the well-known first passage distribution for random walk, for the random variable $N = \inf\{n\colon |\zeta_{n}|=0\}$ we have $P(N>n^2) \geq c/n$ for some possibly smaller $c>0$. Note that although $m_n$ depends on $c$, shrinking $c$ does not affect the estimate. Letting $\tau = \inf\{t\colon|\zeta_t|=0\}$ as before, and letting $a_n = \mu_n/2 = cn/(12b\log(n+1))$, for $a_{n-1}<t<a_n$,
$$t \log t \geq \frac{c(n-1)}{12b\log n}\left( \log(n-1)-\log(12b/c)-\log\log n \right ) \geq \frac{cn}{24b}$$
i.e., $n \leq (24b/c)t\log t$ for $n$ large enough. For the same $t$, then,
$$\P(\tau > t) = \P(|\zeta_t| >0) \geq \P (|\zeta_{t_{n^2}}|>0, t_{n^2}>a_n ).$$
Since the survival time of the cluster is independent of the rate at which it jumps we have
$$\P(\tau >t) \geq c/n( 1 - (bn)^{-1} - O(n^{-2} ) ), \qquad a_{n-1} < t < a_n.$$
The right side is at least $c/(2n)$ for $n$ greater than some $n_0$. Letting $t_0 = a_{n_0}$ and using the upper bound on $n$,
$$\E \tau = \int_0^{\infty}\P(\tau>t)dt \geq \frac{c^2}{48b}\int_{t_0}^{\infty}\frac{1}{t\log t}dt = \infty.$$
\end{proof}

%% file: variants2.tex
\section{Non-backtracking coalescing random walk on trees}
\label{sec:trees}


We are also interested in understanding similar, but less random processes. A lack of symmetry in these settings makes it difficult to apply known techniques. We are hopeful that progress will lead to new ideas.

 The \emph{non-backtracking coalescing random walk} is defined in the same way as the coalescing random walk with particles instead performing non-backtracking random walk. More precisely, the state of a particle is specified by a vertex-edge pair $(u,\{u,v\})$, and when an edge clock rings at a directed edge $(u,w)$, the particle moves to $w$ if and only if $w \neq v$. If the particle moves from $u$ to $w$, its state is updated to $(w,\{w,u\})$, so that its next jump cannot be back to $u$. It will be convenient to assume that each particle is initialized with a uniformly chosen edge along which it cannot move, that is, the particle initially at $v$ has state $(v,\{v,u\})$ where $u$ is a uniform random neighbour of $v$.
 With particles coalescing there is ambiguity about whose path to remember. There are several well-defined ways to assign priority. On a rooted tree, we analyze the special case where we always remember the path of particles moving towards the root. With the model defined in this way we do not quite have a voter model dual, but a closely related process does. Analogous to \thref{thm:main} we prove a necessary and sufficient condition for site recurrence.

\begin{prop} \thlabel{thm:trees} Consider coalescing non-backtracking random walk on a rooted tree with priority given to particles moving towards the root. The process is site recurrent at the root if and only if the expected survival time of the cluster in a certain voter model is infinite.
\end{prop}

We can deduce site recurrence on bounded degree trees and some trees with unbounded degree. 

\begin{thm} \thlabel{cor:trees}
The process from \thref{thm:trees} is site recurrent at the root of either a bounded degree tree or a Galton-Watson tree whose offspring distribution is as in \thref{thm:bgw} (ii).
\end{thm}

Non-backtracking removes a vital symmetry from the argument. The proof goes by, once again, constructing a dual voter process and showing the cluster of the root survives for an infinite expected amount of time. Our ``priority to the root" rule is hand-picked to preserve monotonicity and the existence of a dual voter model. Neither property exists in other equally natural non-backtracking models. Further progress in these different settings will likely require a new approach. Consider the following conjecture:

\begin{conjecture}
Non-backtracking coalescing random walk with any priority scheme is site recurrent on bounded degree trees. 
\end{conjecture}

The inspiration for studying non-backtracking processes comes from the following \emph{meteor model} on $\mathbb R^d$. Place $\epsilon$-balls in Euclidean space with centers according to a unit intensity Poisson process. At time 0 each chooses a direction uniformly randomly and proceeds along this direction at unit speed (non-random). When two meteors collide, they annihilate. 

\begin{conjecture}
The origin a.s.\ is occupied by infinitely many meteors for all $d \geq 1$.
\end{conjecture}

This problem appears quite difficult. It could be discretized to an annihilating system of random walks by uniformly assigning each particle a geodesics to $\infty$ from which it never deviates and steps along according to a Poisson clock. The integer lattice is a natural graph to start with. Or, perhaps hyperbolic space -- in which random walk  paths stay within
a logarithmic neighborhood of a geodesic -- would be a more tractable place to study this problem. 
%

\subsection{Recovering a dual} \label{sec:treedual}

Since the graph is a tree, a particle either moves towards the root, or away, at each jump. Once it has moved away for the first time, at subsequent jumps it must always move away, since the only way back towards the root requires backtracking. To simplify matters, we suppose that at each vertex, the initial forbidden edge is chosen uniformly from the edges that lead away from the root.

Since coalescing particles may have different histories we must decide which one to remember. Therefore, upon collision we define the following rule for annihilating exactly one of the two colliding particles; in the original setting with memoryless walks, any such rule leads to the coalescing model.
 \begin{itemize}
     \item If one particle is moving towards the root, and the other particle is moving away, annihilate the particle moving away.
     \item Otherwise, annihilate the particle currently occupying the site (i.e.\ keep the particle that just moved).
 \end{itemize}

As currently stated, this process depends on past information. Running the process in reverse would require information about the future. Thus, it does not have a dual voter model. Still, a simple observation yields a related model that does have a dual.


Let $X_T$ be the occupation time of the root up to time $T$. Notice that particles moving away from the root are inert; based on the rule above they cannot block upward moving particles, and as noted before they cannot revisit the root. Thus, $X_T$ is unchanged if we delete particles the instant they turn away from the root. This modification gives the following model; for a vertex $v$, let $d_v = \deg v$.
\begin{itemize}
 \item Suppose $v$ is not the root.
 \begin{itemize}
  \item A particle at $v$ moves towards the root at rate $1$.
  \item A particle at $v$ is deleted at rate $d_v-2$.
 \end{itemize}
 \item Suppose $v=\rho$ is the root. A particle at $\rho$ is deleted at rate $d_{\rho}-1$.
\end{itemize}
We can think of this model as follows: each particle attempts to travel up to the root, coalescing with other particles upon collision, and particles (or coalesced collections of particles) are instantaneously zapped out of existence at some rate that depends on their present location.

We can simplify the description somewhat by introducing a single absorbing vertex $\mathfrak{a}$ and considering the process on $V \cup \{ \mathfrak{a}\}$, where every vertex has a directed edge pointing to $v$ which rings at rate described below. 

Particles at $\mathfrak{a}$ do not move. The transitions for particles at $v \in V$ are as follows.
\begin{itemize}
\item Move towards the root at rate $1$.
\item Move directly to $\mathfrak{a}$ at rate $d_v-2$.
\end{itemize}
A graphical representation of this model can be obtained by placing an independent Poisson process with rate $1$ at each upward directed edge, and with rate $d_v-2$ at each vertex $v$.

The model enjoys the same monotonicity as the coalescing random walks -- resetting to the initially full configuration maximizes the probability of occupying the root in the future. Then, as for the coalescing random walks, there is a dual voter model. In this case, deletion of a particle at $v$ corresponds to the addition of $v$ to the cluster of $\mathfrak{a}$. Note that since the direction of motion is reversed in the voter model, clusters on the tree must expand away from the root. Altogether, the voter model has the following transitions. A \emph{down-going} directed edge $(w,u)$ is an edge directed away from $\mathfrak{a}$. These are the rules we use in the proof of \thref{lem:btm}. We box it for emphasis:
\begin{framed}
\begin{itemize}[leftmargin = .8 cm]
\item Along each down-going edge $(w,u)$, at rate $1$, $u$ is added to the cluster containing $w$.
\item At each vertex $v$ on the tree, at rate $d_v-2$, $v$ is added to the cluster of $\mathfrak{a}$.
\end{itemize}
\end{framed}
The existence of a dual lets us establish \thref{thm:trees}.

\begin{proof}[Proof of \thref{thm:trees}]
We have established that $X_T$ in the non-backtracking coalescing random walk has the same distribution as in the simpler model. Since the simpler model has monotonicity and a voter model dual, the same argument used to prove \thref{thm:main} gives the desired equivalence.
\end{proof}

\subsection{Site recurrence}

Now we can turn our attention to proving infinite expected survival time of the root cluster in the voter model. The voter model from Section \ref{sec:treedual} also has the martingale property. As before, let $\zeta_t^v$ denote the cluster that began at $v$.
\begin{lemma}\thlabel{lem:btm}
Consider the voter model $\zeta_t^v \in V \cup \{\mathfrak{a}\}$ described above. For each $v \in V$, so long as $|\zeta_t^v|<\infty$, the size of the cluster $\zeta_t^v$ is a martingale. It transitions to $|\zeta_t| \pm 1$ at rate $\sum_{w \in \zeta_t} (d_w -1) - 2|\{(w,y) \colon w,u \in \zeta_t^v\}|$.
\end{lemma}
\begin{proof}
Fix a vertex $v \in V$ for which we will consider the cluster $\zeta_t = \zeta_t^v$. Let $r_t^+ = r_t^+(v)$  and $r_t^- = r_t^-(v)$ denote the rate at which $|\zeta_t| \rightarrow |\zeta_t| \pm 1$, respectively. Note that the transition rules prohibit $\mathfrak{a} \in \zeta_t$. Moreover, $\zeta_t$ is unchanged if we assume that initially, all vertices but $v$ belong to the cluster of $\mathfrak{a}$. Therefore, it is enough to check that for any finite $W$, if $\zeta_t = W$ and $\zeta_t^\mathfrak{a} = V\cup\{a\}\setminus W$ then $r_t^+ - r_t^-=0$. For a vertex $w \neq \mathfrak{a}$ let $\hat w$ denote its unique parent vertex, i.e., the unique vertex  such that there is a down-going edge to $w$, and let $\mathfrak o(w)$ denote the set of childs vertices. From the transition rules it follows that
$$r_t^+ = \sum_{w \in W}(d_w-1) - |\{u\in \mathfrak o(w)\colon u \in W\}|$$
and
$$r_t^- = \sum_{w \in W}(d_w-2) + \mathbf{1}(\hat w \notin W) = \sum_{w \in W}(d_w-1) - \mathbf{1}(\hat w \in W)$$
and since $\sum_{w \in W}|\{u \in \mathfrak o(w):u \in W\}|$ and $\sum_{w \in W}\mathbf{1}( \hat w \in W)$ are both equal to $|\{(w,u)\colon u,w \in W\}|$,
$$r_t^+ - r_t^- = -\sum_{w \in W}|\{u \in \mathfrak o(w)\colon u \in W\}|+\sum_{w \in W}\mathbf{1}(\hat w \in W) = 0.$$
\end{proof}

\begin{proof}[Proof of \thref{cor:trees}]
With \thref{lem:btm} we can bound the transition rate of $|\zeta_t^\rho|$ by $\sum_{v \in \zeta_t^\rho} d_v.$ To prove the part of \thref{cor:trees} concerning bounded degree trees we can follow the same approach as \thref{thm:bgw} (i); again we use the transition rate bound $D\cdot |\zeta_t^\rho|$. Similarly, we can use the same technique as the proof of \thref{thm:bgw} (ii) to deduce site recurrence for Galton-Watson trees.	
\end{proof}